\pdfoutput=1
\documentclass[10pt]{amsart}
\usepackage[]{hyperref} 
\usepackage{amssymb}
\usepackage[curve]{xypic}
\newcommand{\Omit}[1]{\begin{tiny}#1\end{tiny}}
\renewcommand{\Omit}[1]{}

\newbox\mybox
\def\overtag#1#2#3{\setbox\mybox\hbox{$#1$}\hbox to
  0pt{\vbox to 0pt{\vglue-#3\vglue-\ht\mybox\hbox to \wd\mybox
      {\hss$\scriptstyle#2$\hss}\vss}\hss}\box\mybox}
\def\undertag#1#2#3{\setbox\mybox\hbox{$#1$}\hbox to 0pt{\vbox to
    0pt{\vglue#3\vglue\ht\mybox\hbox to \wd\mybox
      {\hss$\scriptstyle#2$\hss}\vss}\hss}\box\mybox}
\def\lefttag#1#2#3{\hbox to 0pt{\vbox to 0pt{\vss\hbox to
      0pt{\hss$\scriptstyle#2$\hskip#3}\vss}}#1}
\def\righttag#1#2#3{\hbox to 0pt{\vbox to 0pt{\vss\hbox to
      0pt{\hskip#3$\scriptstyle#2$\hss}\vss}}#1}

\def\Dot{\lower.2pt\hbox to 3.5pt{\hss$\bullet$\hss}}
\def\Circ{\lower.2pt\hbox to 3.5pt{\hss$\circ$\hss}}

\def\splicediag#1#2{\xymatrix@R=#1pt@C=#2pt@M=0pt@W=0pt@H=0pt}

\renewcommand\frame[2][3pt]{\hbox{$\vcenter{\hbox{\vrule\vbox
{\hrule\kern#1\hbox{\kern#1$#2$\kern#1}\kern#1\hrule}\vrule}}$}}

\newcommand\lineto{\ar@{-}}
\newcommand\dashto{\ar@{--}}
\newcommand\dotto{\ar@{.}}
\newcommand{\C}{\mathbb C}
\newcommand{\R}{\mathbb R}
\newcommand{\E}{\mathbb E}

\newcommand{\Z}{\mathbb Z}
\newcommand{\Q}{\mathbb Q}

\newtheorem*{theorem*}{Theorem}
\newtheorem{theorem}{Theorem}[section]
\newtheorem{theoremi}{Theorem}
\newtheorem{proposition}[theorem]{Proposition}
\newtheorem{lemma}[theorem]{Lemma}
\newtheorem{corollary}[theorem]{Corollary}
\newtheorem*{corollary*}{Corollary}
\theoremstyle{definition}
\newtheorem{example}[theorem]{Example}
\newtheorem*{example*}{Example}
\newtheorem{definition}[theorem]{Definition}
\newtheorem*{definition*}{Definition}
\newtheorem{remark}[theorem]{Remark}
\newtheorem*{remark*}{Remark}

\begin{document}
\title{Orbifold splice quotients and  log covers of surface pairs}
\author{Walter D. Neumann} \thanks{Neumann's research supported under NSF grant
  no.\ DMS-1608600} \address{Department of
  Mathematics\\Barnard College, Columbia University\\New York, NY
  10027} \email{neumann@math.columbia.edu} \author{Jonathan Wahl}
\address{Department of Mathematics\\The University of North
  Carolina\\Chapel Hill, NC 27599-3250} \email{jmwahl@email.unc.edu}
  \begin{abstract} A three-dimensional orbifold $(\Sigma, \gamma_i, n_i)$, where $\Sigma$ is a rational homology sphere,  has a universal abelian orbifold covering, whose covering group is the first orbifold homology.  A singular pair $(X,C)$, where $X$ is a normal surface singularity with $\Q$HS link and $C$ is a Weil divisor, gives rise on its boundary to an orbifold.  One studies the preceding orbifold notions in the algebro-geometric setting, proving the existence of the universal abelian log cover of a pair.  A key theorem computes the orbifold homology from an appropriate resolution of the pair.  In analogy with the case where $C$ is empty and one considers the universal abelian cover, under certain conditions on a resolution graph one can construct pairs and their universal abelian log covers.  Such pairs are called orbifold splice quotients.  
  \end{abstract}
\maketitle

Let $(X,0)$ be the germ of a normal complex surface singularity whose link $\Sigma$ is a rational homology sphere  ($\mathbb
Q$HS).  The topology of $\Sigma$ is determined from any good resolution $(\tilde
X,E)\rightarrow (X,0)$ by its weighted dual resolution graph
$\Gamma$, which is a tree.  
The discriminant group $D(\Gamma)$, the cokernel of the intersection matrix $(E_i\cdot E_j)$, is isomorphic to the first homology group of $\Sigma$.  The universal abelian cover (UAC) $\Sigma'\rightarrow \Sigma$ extends to a ``cover'' $(X',0)\rightarrow (X,0)$ of singularities, also called the UAC.  These covers are quotients by an action of $D(\Gamma)$.

It was shown in \cite{nw1} that under some mild conditions on a graph $\Gamma$ with $t$ ends (``semigroup and congruence  conditions''), one can construct 
\begin{enumerate}
\item  explicit classes of complete intersection singularities $(X',0)\subset (\C^t,0)$
\item a diagonal representation $D(\Gamma)\subset (\C^*)^t$ acting freely on each $X'-\{0\}$
\item $(X',0)\rightarrow (X',0)/D(\Gamma)\equiv (X,0)$, the UAC of a singularity with graph $\Gamma$. 
\end{enumerate}
Such $X$, called ``splice quotient singularities,'' are thus explicit examples having a resolution with the given graph $\Gamma$.  (See (6.1) below for a precise statement).  For their role in singularity theory, see e.g. \cite{no} and \cite{no2}.  

The ``end-curve theorem'' of \cite{nw3} characterizes those singularities $(X,0)$ which are splice quotients.  The $t$ ends of the graph $\Gamma$ of the minimal good resolution correspond to $t$ isotopy classes of knots in the link $\Sigma$.  Splice quotient singularities are exactly those for which each such class is represented (up to a  multiple) by the zero-set of a function on $(X,0)$.  In this way, one recovers the result (\cite{okuma2}) that rational singularities, and those minimally elliptic singularities with $\Q$HS link, are splice quotients.


Rather than consider only covering spaces of a three-manifold $\Sigma$, it is natural to consider \emph{orbifold covers}.   Given a class of $r$ knots $\{\gamma_i\}$ in $\Sigma$, with multiplicities $n_i\geq 1$, consider ``covers'' which are allowed to branch over these knots, with ramification indices bounded by the $n_i$.   (Compare with looking at extensions of a number field allowed to ramify over some finite set of primes.)  In Section 1 this notion is discussed and related to usual orbifold language, including
the \emph{universal abelian orbifold cover} (UAOC) $\Sigma''\rightarrow \Sigma$ and its covering group, the \emph{(first) orbifold homology group} $H_1^{orb}(\Sigma)$.  When $\Sigma$ is a $\Q$HS,  Proposition 1.3 indicates how to compute $H_1^{orb}(\Sigma)$, and proves it is an extension of the usual first homology group of $\Sigma$ by a product of cyclic groups of order $n_i$.  Abelian orbifold covers are classified by quotients of $H_1^{orb}(\Sigma)$.

There is an analogous notion of cover for a singular pair $(X,C)$.  Here $(X,0)$ is again a normal surface singularity with link $\Sigma$, and  $C=\Sigma_{i=1}^r n_iC_i$ is a sum of irreducible Weil divisors $\{C_i\}$.  By intersecting with a neighborhood boundary, a pair $(X,C)$ gives rise to $\Sigma$ and the orbifold data of knots $\gamma_i$ and weights $n_i$.  A \emph{log cover} of the pair is a finite covering of $X$ branched only over the $C_i$, with ramification bounded by $n_i$.  Basic results about branched cyclic covers are given in Section 2.

For $(X,C)$ with $\Q$HS link, Theorem 3.1 proves the existence of the \emph{universal abelian log cover}, or \emph{UALC}, a log cover of pairs $(X'',C'')\rightarrow (X,C)$ which corresponds to the UAOC on the boundary.   Note that the $C_i$ need not 
be $\Q$-Cartier.

From the graph of any good resolution of $(X,0)$,  one can compute the covering group $D(\Gamma)$ of the UAC $(X',0)\rightarrow (X,0)$.  To find the 
covering group of the UALC of a pair $(X,C)$, consider the smallest good resolution $(\tilde{X},E)\rightarrow (X,0)$ for which the proper transform $\tilde{C}_i$ of each $C_i$ intersects transversally a leaf (or end) $E_i$ of $E$, and each leaf intersects at most one $\tilde{C}_i$.  Such leaves are called \emph{special}, and $\tilde{X}\equiv\tilde{X}_C$ is the \emph{minimal orbifold resolution} of $(X,C)$.  
Denote the corresponding graph $\Gamma_C$; it is quasi-minimal (see (6.1) below).  Now let $\Gamma_C^*$ be the same graph, but decorated by the placement of an arrow and a weight $n_i$ at each special leaf $E_i$.  Finally,  let $D(\Gamma_C^*)$ be the cokernel of the matrix $(E_i\cdot E_j)$ modified by multiplying each row corresponding to a special leaf $E_i$ by $n_i$ .  One can easily prove there is a short exact sequence 
   $$0\rightarrow \oplus \mathbb Z/(n_i) \rightarrow D(\Gamma_C^*)\rightarrow D(\Gamma_C)\rightarrow 0.$$

\begin{example*} If $X=\mathbb C^2$ and $C=\{x^py^q(x-y)^r(y^2-x^3)^s=0\}$, the diagram for $\Gamma_C^*$ is 
$$
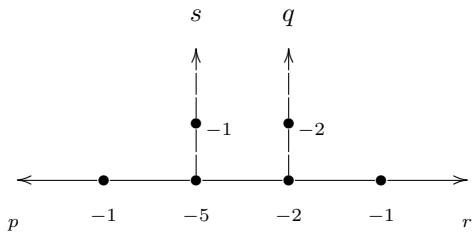
\xymatrix@R=8pt@C=30pt@M=0pt@W=0pt@H=0pt{
&&&&{s}&{q}\\
&&&&&&&&&\\
&&&&\lineto[d]\ar[u]&\lineto[d]\ar[u]&\\
&&&&\lineto[d]&\lineto[d]&\\
&&&&\righttag{\bullet}{-1}{6pt}&\righttag{\bullet}{-2}{6pt}&&&\\
&&&&\lineto[u]&\lineto[u]&\\
&&\undertag{~}{p}{12pt}&\ar[l]\undertag{\bullet}{-1}{4pt}\lineto[r]&\undertag{\bullet}{-5}{4pt}\lineto[r]\lineto[u]&\undertag{\bullet}{-2}{4pt}\lineto[r]\lineto[u]&\undertag{\bullet}{-1}{4pt}\lineto[r]&\undertag{\ar}{r}{12pt}\\
&&\\
&&\\
&&\\
}
$$ 
One finds $\Gamma_C$ by removing the arrows and $p,q,r,s$.  $D(\Gamma_C^*)$ is the direct sum of $4$ cyclic groups, of orders $p,q,r,s$ respectively.
\end{example*}
The first important result is proved in $(4.3)$:
\begin{theoremi} $D(\Gamma_C^*)$, the cokernel of the modified intersection matrix above, is isomorphic to $H_1^{orb}(\Sigma)$, and is the covering group of the UALC of $(X,C)$.
\end{theoremi}

According to \cite{nw1}, a graph $\Gamma$ with $t$ ends gives a faithful diagonal representation
 $D(\Gamma)\hookrightarrow (\C^*)^t$.  In Proposition 5.3, we show there is a representation $D(\Gamma^*)\hookrightarrow (\C^*)^t$ compatible with that of $D(\Gamma)$ via a map $N:(\C^*)^t\rightarrow (\C^*)^t$ which raises coordinates to the power given by the weight at the corresponding special end of $\Gamma^*$, and leaves other entries fixed.
 
Our purpose is to define what it means for a pair  $(X,C)$  to be an \emph{orbifold splice quotient}.   The graphs $\Gamma_C$ and $\Gamma_C^*$ depends only on the boundary topology of the corresponding orbifold.  In $(6.5)$ we have
  
\begin{theoremi}  Let $\Gamma$ be a quasi-minimal graph with $t$ ends and $\Gamma^*$ a decorated version, assigning a weight to the special ends.  Let $D(\Gamma^*)\hookrightarrow (\C^*)^t$ be the corresponding representation. If $\Gamma$ satsifies the semigroup and congruence conditions, then  

\begin{enumerate}  \item there are isolated complete intersection singularities $(X'',0)\subset (\C^t,0)$ 
\item on each $(X'',0)$, there is a Cartier divisor $C''$ so that $D(\Gamma^*)$ acts on the pair $(X'',C'')$, freely off the support of $C''$
\item each quotient $(X'',C'')\rightarrow(X'',C'')/D(\Gamma^*)\equiv (X,C)$ is the UALC
\item the orbifold data associated to $(X,C)$ is $\Gamma, \Gamma^*$.
\end{enumerate}
\end{theoremi}
\begin{definition*}  A pair $(X,C)$ in the above Theorem is called an orbifold splice quotient.
\end{definition*}

We outline the construction and proof.  
 Assign a variable $x_i$ to every end of $\Gamma$.   By the semigroup and congruence hypotheses and the major Theorem 7.2 of \cite{nw1}, one may consider a set of $D(\Gamma)$-invariant splice-equations $\{f_j(x_i)=0\}$, giving an isolated complete intersection singularity $(X',0)\subset (\C^{t},0)$, which is the UAC of the quotient $X\equiv X'/D(\Gamma).$  
The Cartier divisor $\{x_i=0\}$ on $X'$ is a reduced curve; some power of $x_i$ is $D(\Gamma)$-invariant, and $x_i^{det\  \Gamma}$ provides a Cartier divisor on $X$ whose reduction is an irreducible $\mathbb Q$-Cartier divisor.  In particular, for each special end of $\Gamma$ and corresponding $x_i$, define the $\mathbb Q$-Cartier divisor $C_i$ as the reduced image of $\{x_i=0\}$ on $(X,0)$.  
       
        Next, replace a variable $x_i$ corresponding to a special leaf by $z_i^{n_i}$.   Consider the new complete intersection singularity $(X'',0)\subset (\C^t,0)$ defined by $\{f_j(z_i^{n_i})=0\}$.  This is the inverse image of $(X',0)$ under the aforementioned power map $N:\C^t\rightarrow \C^t$ which raises to the $n_i^{th}$ power in the $i^{th}$ special entry, and is the identity on others.   We prove that $(X'',0)$ is an isolated complete intersection singularity.  We also show  (5.3)  $D(\Gamma^*)$ is the inverse image under $N$ of $D(\Gamma)$, hence acts on $(X'',0)$.    Letting $C''$ be the sum of the divisors $z_i=0$ for the special leaves, the result will follow.

 There is an analogue of the End-Curve Theorem of \cite{nw3} for orbifold splice quotients.  For a given pair $(X,C)$, consider the minimal orbifold resolution $(\tilde{X}_C,E)\rightarrow (X,0)$ and the corresponding graph $\Gamma_C$, with the special ends noted.  In $(6.8)$ we prove:
 
 \begin{theoremi} (Orbifold End-Curve Theorem) A pair $(X,\Sigma n_iC_i)$ is an orbifold splice quotient if and only each $C_i$ is $\Q$-Cartier, and for every non-special end of the graph $\Gamma_C$, there is a function on $(X,0)$ whose zero-set on $\tilde{X}_C$ cuts out a smooth curve transversal to the end. 
 \end{theoremi}
 
 It follows from the definition that an orbifold splice quotient  is a splice quotient for which each $C_i$ is $\Q$-Cartier.  We do not know whether the converse is true; we suspect not. 
 However, in $(6.9)$ and $(6.10)$ we have the following:
 
 \begin{corollary*} If $(X,0)$ has a rational singularity, then every pair $(X,C)$ is an orbifold splice quotient.
 \end{corollary*}
 
 \begin{corollary*}  Let $(X,C)$ be a pair for which every non-special end of $\Gamma_C$ comes from an end of the minimal good resolution.  Then $(X,C)$ is an orbifold splice quotient if and only if $(X,0)$ is a splice quotient and every $C_i$ is $\Q$-Cartier.
 \end{corollary*}

 The paper is organized as follows: In Section 1 we recall the basics of orbifolds in dimension three, the universal orbifold covering and fundamental group, and orbifold homology.  Section 2 introduces singular pairs and their log covers, with focus on cyclic covers.  From now on, one assumes the link is a $\Q$HS.  Section 3 gives an algebro-geometric construction of the universal abelian log cover of a pair.  The orbifold homology is computed in Section 4 from a plumbing diagram; the crucial representation of the orbifold discriminant group occurs in Section 5.  The definition and construction of orbifold splice quotients appears in Section 6, with an example presented in Section 7.
 
 We are grateful to Eduard Looijenga and Helge M\o{}ller Pedersen for assistance with proving Proposition 4.4.
 \bigskip

\section{{Orbifolds}}

For the special purposes of this paper, an \emph{orbifold} will be a compact
oriented $3$--manifold $\Sigma$, plus a disjoint set of 
embedded circles (or ``knots'') $\gamma_1,\cdots,\gamma_r\subset
\Sigma$, to which are associated positive integers $n_1,\cdots,n_r$ (the \emph{orbifold weights}).  While this is not the standard definition of orbifold (e.g., [5]), it will be useful in considering branched covers with restricted ramification.  See Example $1.2$ below.

\begin{definition}
A (finite) \emph{orbifold covering} consists of another orbifold
$(\Sigma', \gamma '_j, n'_j)$ and a finite map $f:\Sigma'\rightarrow
\Sigma$, with the following properties:
\begin{enumerate}
\item the inverse image of a knot $\gamma_i$ consists of various
  $\gamma'_{i_j}$, and $n'_{i_j}$ dividing $n_i$
\item $f^{-1}(\cup \gamma_i)=\cup\gamma'_j$, and $f$ is a covering map
  off this set
\item locally around $\gamma'_{i_j}$, $f$ maps to a neighborhood of
  $\gamma_i$ as an orbifold quotient of order $n_i/n'_{i_j}$.
\end{enumerate}
\end{definition}
The last statement means that if the pairs $(\Sigma,\gamma_i)$ and
$(\Sigma',\gamma'_{i_j})$ are given locally as $(\R^3,
z-\text{axis})$, then the map $f$ is given by dividing out by the group of
rotations of order $n_i/n'_{i_j}$ around the $z$-axis.  Since $f$ is a
covering map of some degree $d$ off the $\gamma'_{i_j}$, it follows
that for each $i$, $$\sum_{j} n_i/n'_{i_j} =d.$$

More precisely, the orbifolds we are considering could be described as
``locally cyclic orientable $3$-orbifolds'' since their local
structure is always $D^3/(\Z/(n))$ for some $n$ ($D^3$ a three-disk). We
restrict to cyclic groups since all our $3$--manifolds are
boundaries of $4$-manifolds with complex structure. 
\begin{example}  Let $M'$ be a compact oriented 3-manifold on which a finite group $G$
acts smoothly and faithfully, preserving orientation, and with
the property that the isotropy group $G_x$ at any point $x$ is
cyclic.   We make $M'\rightarrow M\equiv M'/G$ an orbifold cover. The knots $\gamma_i\subset M$ are the images of the components of the set of
points on $M'$ with non-trivial isotropy, and the weight $n_i$ is the order of
this isotropy group. The $\gamma'_{i_j}$ are the inverse images of the $\gamma_i$, and the $n'_{i_j}=1$.
\end{example}

From now on we will often use the notation $\Sigma$ both for the orbifold
and its underlying space when what  is meant is clear from context.  Most of the definitions here and below are due to Thurston (e.g., Chapter 13 of his notes \cite{t}).

An orbifold has an \emph{orbifold fundamental group}
$\pi_1^{orb}(\Sigma)$, obtained by quotienting
$\pi_1(\Sigma -\bigcup_i\gamma_i)$ by the relations 
$\mu_i^{n_i}=1$, where $\mu_i$ is represented by the boundary of a
small transverse disc to $\gamma_i$. This can also be described as
$\pi_1^{orb}(\Sigma)=\pi_1(\Sigma')$, where $\Sigma'$ is the result of
replacing a $D^2\times S^1$ neighborhood of each $\gamma_i$ by
$K(\Z/(n_i), 1)\times S^1$, glued so that $\mu_i$ represents a
generator of $\pi_1(K(\Z/(n_i),1))$. (Here $K(\Z/(n_i),1)$ is an
Eilenberg-Maclane complex---a CW-complex with fundamental group
$\Z/(n_i)$ and contractible universal cover.)

The Galois correspondence between covering spaces and subgroups of the
fundamental group extends to orbifolds, using orbifold covers and
orbifold fundamental group (see e.g., \cite{nn}). In particular, any
orbifold has a \emph{universal abelian orbifold cover}, or
\emph{UAOC}, classified by the commutator subgroup of
$\pi_1^{orb}(\Sigma)$. Its covering transformation group is the
abelianization $\pi_1^{orb}(\Sigma)^{ab}$ of $\pi_1^{orb}(\Sigma)$;
this is also $H_1^{orb}(\Sigma;\Z):=H_1(\Sigma';\Z)$, the degree $1$
orbifold homology of $\Sigma$ (see \cite{nn},  \cite{ped}, and \cite{sc} for more on orbifold
homology).

Not every orbifold is the quotient of a manifold by a finite group, as in Example $2.1$; an
example is $(S^2\times S^1, \{p\}\times S^1, n)$ for $n>1$.  However,  if
$\Sigma$ is a $\Q$HS we prove:

\begin{proposition}\label{prop:UAOC} 
  Let $(\Sigma, \gamma_i, n_i)$ be an orbifold for which the
  underlying space $\Sigma$ is a $\Q$HS.  Then the \emph{UAOC}
  $\widetilde\Sigma^{ab}$ of $\Sigma$ is a finite cover of $\Sigma$,
  and its covering group $G=H_1^{orb}(\Sigma;\Z)$ sits in an exact
  sequence
$$0\rightarrow \Z/(n_1)\oplus \cdots \oplus 
\Z/(n_r)\rightarrow G \rightarrow
H_1(\Sigma;\Z)\rightarrow 0\,.$$ 
\end{proposition}

The Proposition follows from two Lemmas.

\begin{lemma}  If $\Sigma$ is a $\Q$HS, there is an exact sequence 
$$0\rightarrow \bigoplus_{i=1}^r \Z\rightarrow H_1(\Sigma-\bigcup_i\gamma_i)\rightarrow H_1(\Sigma)\rightarrow 0,$$
where the first map is dotting with $(\mu_1,\cdots,\mu_r)$.  
\begin{proof} The long-exact sequence of the pair $(\Sigma, \Sigma-\bigcup_i\gamma_i)$ gives
$$0=H_2(\Sigma)\rightarrow H_2(\Sigma, \Sigma-\bigcup_i\gamma_i)\rightarrow
 H_1(\Sigma-\bigcup_i\gamma_i)\rightarrow H_1(\Sigma)\rightarrow H_1(\Sigma, \Sigma-\bigcup_i\gamma_i)=0.$$
The first term is $0$ since $\Sigma$ is $\Q$HS.  Let $T=\cup_iT_i$ be the union of small closed tubular neighborhoods of the $\gamma_i$, each one a closed 2-disk bundle over an $S^1$; let $T^o$ and $T_i^o$ denote their interiors.  Then the inclusion
 $\Sigma-\bigcup_i\gamma_i \subset \Sigma-T^o$ induces an isomorphism 
 $H_j(\Sigma, \Sigma-\bigcup_i\gamma_i))\cong H_j(\Sigma, \Sigma - T^o)$.  By excision (removing $\Sigma-T$), one has
$$H_j(\Sigma, \Sigma-T^o)\cong H_j(T,T-T^o))=H_j(T,\partial T)=\oplus_iH_j(T_i,\partial T_i).$$
But $H_2(T_i,\partial T_i)=\ker H_1(\partial T_i)\rightarrow H_1(T_i)$ is the free $\Z$-module generated by $\mu_i$, while $H_1(T_i,\partial T_i)=0$, whence the result. 
\end{proof}
\end{lemma}
\begin{lemma} $H_1^{orb}(\Sigma)\cong H_1(\Sigma-\bigcup_i\gamma_i))/(n_1\mu_1,\cdots, n_r\mu_r).$
\begin{proof}  If $G$ is a group, $N$ a normal subgroup, then abelianizing $G/N$ is the same as abelianizing $G$ and then modding out by the image of $N$ in $G_{ab}$; either method gives $G/G'N$, where $'$ denotes commutator subgroup.  For, every element of the commutator subgroup is a product of commutators, so the map $G'\rightarrow (G/N)'$ is surjective.  Apply now to 
$G=\pi_1(\Sigma -\bigcup_i\gamma_i)$ and $N$ the normal subgroup generated by the relations $\mu_i^{n_i}=1$.  Then $G/N=\pi_1^{orb}(\Sigma),$ and abelianizing gives $H_1^{orb}(\Sigma)$; $G_{ab}=H_1(\Sigma-\bigcup_i\gamma_i),$ and the image of $N$ is $(n_1\mu_1,\cdots, n_r\mu_r)$ . 
\end{proof}

\end{lemma}
\begin{remark}
 The Proposition shows one could construct the UAOC by first forming
the UAC $\Sigma'\rightarrow \Sigma$, taking the inverse images of the $\gamma_i$'s, and then making a sequence $\widetilde\Sigma^{ab}\rightarrow \Sigma'$ of cyclic branched covers over these knots.  The orbifold weights on  $\widetilde\Sigma^{ab}$ are equal to $1$.

\end{remark}

  \bigskip

\section{{Covers of singular pairs}}   
The singularity version concerns \emph{singular pairs} $(X,C)$.  Here,
$(X,0)$ is the germ of a complex normal surface singularity (thus
homeomorphic to the cone over its link $\Sigma$), and $C=\sum_{i=1}^r n_iC_i$ is
a positive combination of reduced and irreducible Weil divisors.  Intersecting with the boundary of a small ball neighborhood of the
singular point in a smooth ambient space gives an orbifold link in a
natural way.  

The data of a singular pair $(X,C)$ is more frequently encoded in the literature as a \emph{log pair} $(X,\Delta_X)$, with the \emph{boundary of $X$}  the effective $\Q$-divisor $$\Delta_X=\sum_{i=1}^r(1/n_i)C_i.$$  Log pairs are essential notions in algebraic geometry (e.g. \cite{kol}), and $1/n_i$ could be replaced by any real number between $0$ and $1$.  However, for this paper the singular pair $(X,\sum n_iC_i)$ is a more convenient way to display branching data for branched covers.

We denote the \emph{support of $C$} by $|C|=\cup_{i=1}^rC_i$.
 A \emph{map of pairs}  $f:(X',C')\rightarrow (X,C)$ is a finite map of normal germs $f:X'\rightarrow X$ so that $f(|C'|)=|C|.$ 

\begin{definition}A \emph{log cover of pairs} $f:(X',C')\rightarrow (X,C)$
is a map of pairs with the
following additional properties:
\begin{enumerate}
\item $f^{-1}(|C|)=|C'|$, and $f$ is an unramified covering
  space off this set.
\item For $C=\sum_{i=1}^rn_iC_i$ and  $f^*C_i=\sum_{j=1}^{s_i}m_{ij} C'_{ij}$ (as Weil divisors), one has $m_{ij}|n_i$, all $j$.
\item $C'=\sum_{i=1}^r \sum_{j=1}^{s_i}n_i/m_{ij}C'_{ij}$.

\end{enumerate}
\end{definition}
In other words, for a log cover of pairs, the map $f:X'\rightarrow X$ is ramified only over the $C_i$, and the ramification index  $m_{ij}$ on any curve above it divides $n_i$.  (In particular, if $n_i=1$, no ramification is allowed above $C_i$.)  

\begin{remark}  Note that if a map of pairs satisfies $(1)$, then the conditions $(2)$ and $(3)$ above are together equivalent to a relation on the boundary divisors, namely $$f^*(\Delta_X)=\Delta_{X'}.$$
\end{remark}

We note the easy

\begin{lemma}  Let $f:X'\rightarrow X$ be a finite map of germs of normal surface singularities.  If $f$ is ramified off the singular points, then there are unique minimal divisors $C'$ on $X'$ and $C$ on $X$ so that $f:(X',C')\rightarrow (X,C)$ is a log cover of pairs.
\begin{proof} Let $C_1, \cdots, C_r$ be the irreducible components of the branch curve of $f$ on $X$.  Write $f^*C_i=\sum_{j=1}^{s_i}m_{ij} C'_{ij}$ as before.  Define $n_i=lcm(m_{ij}, j=1,\cdots,s_i)$.  Then set $C=\sum_{i=1}^rn_iC_i$ and  $C'=\sum_{i=1}^r \sum_{j=1}^{s_i}n_i/m_{ij}C'_{ij}$.  The result now follows.
\end{proof}
\end{lemma}

       A  log cover of pairs $(X',C')\rightarrow (X,C)$  is said to be \emph{abelian}  (resp., \emph{cyclic}) if an abelian (resp., cyclic)  group $G$ acts on $X'$ with quotient $X$, permuting the curves lying above each $C_i$.  
               
       To make a cyclic cover,  take the $n$th root of an appropriate function $h$ on $X$, and normalize.  Let $A$ be the analytic local ring of $X$, $h\in m_A$ non-$0$.  Then $T^n-h$ is an irreducible polynomial over the quotient field of $A$ iff $h$ is not a $d$th power, where $d|n$.  One may verify this condition as follows: if the divisor $(h)$ of zeroes of $h$ is the sum of irreducible Weil divisors $\sum_{i=1}^rk_iC_i$, then for any $d>1$ dividing $n$ we have that  $(1/d)\sum_{i=1}^r k_iC_i$ is \emph{not} a principal divisor.  In this case, adjoining an $n$th root of $h$ and normalizing gives an $n$-cyclic cover of pairs $(X',C')\rightarrow (X,C)$, for appropriate $C$ and $C'$.  Specifically, examining the cyclic behavior over each $C_i$ and normalizing, we find that $m_{ij}=n/gcd(n,k_i)=n_i$, so that $$C=\sum_{i=1}^r(n/gcd(n,k_i))C_i,\ \ \  C'=\sum_{i=1}^r\sum_{j=1}^{gcd(n,k_i)}C'_{ij}$$ is as in Lemma 3.2.  If $n|k_i$, then the coefficient of $C_i$ is $1$, so that no branching occurs over this curve and it need not be part of $C$.  We summarize in the 
       
       \begin{lemma}  Let $(X,0)$ be a normal germ, $h\in m_A$ a non-$0$ function with divisor of zeroes $\sum_{i=1}^r k_iC_i.$  Suppose $n>1$ is such that $h$ is not a $d$th power, for any $d>1$ which divides $n$.  Then adjoining an $n$th root of $h$ and normalizing gives an $n$-cyclic cover of pairs $(X',C')\rightarrow (X,C),$ with $$C=\sum\\' (n/(gcd(n,k_i))C_i,$$  where $'$ means the sum is taken over those $C_i$ for which $n$ does not divide $k_i$.
       \end{lemma}   
   
   We illustrate the last construction by asking whether there exists an $n$-cyclic cover totally ramified over a given irreducible Weil divisor.  Recall that Cl($X$), the
\emph{divisor class group of $X$}, is the free abelian group
generated by the irreducible Weil divisors on $X$ (or equivalently
height one prime ideals in $A=\mathcal O_{X,0}$), modulo principal
divisors.  
   
   \begin{proposition}  Let $(X,0)$ be a normal surface singularity, $C_1$ an irreducible Weil divisor, $n>1$ a positive integer.  
   \begin{enumerate} 
   \item There exists an n-cyclic cover of pairs $(X',C')\rightarrow (X,nC_1)$ with $C'$ irreducible if and only if the class of $C_1$ in Cl($X$) is divisible by $n$.
   \item Such a cover is unique exactly when the link of $X$ is a $\Q$HS and the discriminant group has order prime to $n$.
      \end{enumerate}
   \begin{proof}
      As explained above, suppose one has an $n$-cyclic cover obtained from a function whose zero divisor is $\sum_{i=1}^rk_iC_i$.  It is branched exactly over those $C_i$ for which $n$ does not divide $k_i.$  To have  a cover as in $(1)$ requires therefore that an $h\in A$ exist with $(h)=k_1C_1+nD$, where  $k_1$ and $n$ are relatively prime. If $uk_1$ is congruent to $1$ mod $n$, we can use $h^u$  with $(h^u)=C_1+nD'$.  Thus, the class of $C_1$ is divisible by $n$.  
      
      Conversely, if the class of $C_1$ is divisible by $n$, then for some $h$ in the quotient field of $A$ and divisor $G$, we can write $$C_1=nG+(h).$$ Writing $h=j/k$, with $j,k \in A$, we have that $(jk^{n-1})=C_1+nG'$, where $G'=(k)-G$. So, adjoining an $n^{th}$ root of the regular function $(jk^{n-1})$  yields a cyclic cover which (by Lemma 2.4) is branched only over $C_1$.
      
         
      As for uniqueness, recall that adjoining $n$th roots of two elements $h$ and $h'$ gives the same field extension if and only if $h/h'$ is an $n$th power.  Suppose there is an effective non-Cartier divisor $C$ so that $nC$ is Cartier, with $(\alpha)= nC$.  Then the Cartier divisors of $h$ and $h\alpha$ differ by a multiple of $n$; but extracting $n$th roots gives different field extensions.  So, the uniqueness question is exactly whether there are non-trivial elements of order dividing $n$ in the divisor class group.  If the link of $X$ is not a $\Q$-homology sphere, then the divisor class group contains a vector space modulo a free $\Z$-module, hence contains elements of all finite order.  If the link is a $\Q$HS, then as recalled in the next section the divisor class group is a direct sum of a complex vector space plus the (finite) discriminant group.  The result follows.
   \end{proof}
\end{proposition}   
We illustrate the last Proposition with the following
  \begin{example} Consider the $A_1$ singularity
  $X=\{(x,y,z)\in\C^3:xz-y^2=0\}$, $C_1$  the line $\{y=z=0\}$, and $C_2$ the irreducible Cartier divisor $\{z-x^2=0\}$.
   The link of $(X,0)$ is the lens space
  $L(2,1)=\R P^3$ with fundamental group $\Z/(2)$, and the divisor class group is cyclic of order $2$, generated by the class of $C_1$.  The principal divisor $(z)=2C_1$, and $z$ is not a power of any element.  We consider $n$th root constructions of $z$ and $z(z-x^2)$.
      \begin{enumerate}
   \item For $n=2k+1$,  the $n$th root of $z$ gives $(X',C'_1)\rightarrow (X, (2k+1)C_1)$.   $X'$ is another $A_1$ singularity $\{xv-u^2=0\}$, $C_1'$ is the line $\{u=v=0\}$, and the coordinates are related by $y=uv^k, z=v^{2k+1}.$  The group action on $X'$ is given by $(x,u,v)\mapsto (x,\zeta u, \zeta^2 v)$, where $\zeta$ is a primitive $(2k+1)$st root of $1$.
    \item  For $n=2k$, the $n$th root of $z$ gives $(\C^2,\{v=0\})\rightarrow (X,kC_1)$.  A generator of the covering group acts via $(u,v)\mapsto (-u,\zeta v)$, where $\zeta$ is a primitive $(2k)$th root of $1$.  Thus, $x=u^2, y=-uv^k, z=v^{2k}.$  The cover factors through the UAC, given by taking the square root of $z$ (the case $k=1$).
    \item The square roots of $z-x^2$ and $z(z-x^2)$ give (Lemma $2.4$) non-isomorphic double covers of $(X,2C_2)$.  In the first case, the cover is obtained by adjoining $T$ with $T^2=z-x^2$, giving a $D_4$ singularity $y^2=x(T^2+x^2)$.  For the second cover, one adjoins $S^2=z(z-x^2)$; normalize via $P=Sx/y$, since $P^2=x(z-x^2)$, and note also that $SP=y(z-x^2)$.  This second cover, of embedding dimension $5$, is an $8/3$ cyclic quotient singularity.   Example $3.4$ below gives the full story.
      \end{enumerate}
\end{example}

\section{{universal abelian log cover of a pair}}

If $(X,0)$ has $\Q$HS link $\Sigma$, then for topological reasons the UAC on the boundary extends to a UAC of $X-\{0\}$, so taking integral close of $X$ in this extension gives what we call the  UAC $(X',0)\rightarrow (X,0)$.  One can also construct $(X',0)$ directly using a covering of a good resolution of $X$, as in \cite{okuma}, (3.1).

 If $(X,C)$ is a  singular pair for which $X$ has $\Q$HS link $\Sigma$, one might in principle try to extend a cover of the corresponding orbifold associated to $\Sigma$ to a branched cover of $X-\{0\}$ and then of $X$.  It seems more natural to make a direct algebro-geometric construction.  We have the following definition/theorem.
 
\begin{theorem}  Suppose $(X,C)$ is a singular pair for which $X$ has $\Q$HS link $\Sigma$.  Then there exists a  \emph{universal abelian log
    cover} (or \emph{UALC})  $(X'',C'')\rightarrow(X,C)$, a   
cover of pairs which induces the UAOC on the boundary.
\end{theorem}



\begin{proof}  Let  $(\tilde{X},E)\rightarrow
(X,0)$ be a good resolution, and $\E$ the lattice generated by the
irreducible exceptional curves $E_i$. 
Following Mumford \cite{mum}, one has a natural direct
sum decomposition
$$Cl(X,0)\cong H^1(\tilde{X},\mathcal O_{\tilde{X}})\oplus \E^*/\E,$$
the direct sum of a complex vector space of dimension $p_g(X)$ plus
the finite \emph{discriminant group} isomorphic to $H_1(\Sigma;\Z)$, of order $|det(E_i\cdot E_j)|$
(see \cite{okuma}, (2.3) for a nice discussion.)  The universal abelian cover
$g:(X',0)\rightarrow (X,0)$ is unramified off the singular point, 
with covering group $\E^*/\E$.  We indicate how to form the UALC via a
sequence of cyclic covers of $X'$.  

Write $C=\sum n_iC_i$ (we may as well assume all $n_i>1$).
Since a vector space is a divisible group, we may write the class of
$C_i$ in Cl($X$) uniquely as $[C_i]=n_i[D_i]+t_i$, where $t_i$ is a torsion
divisor class.  But the class of $g^*(t_i)$ is trivial, i.e.,
$g^*[C_i]-n_ig^*[D_i]$ is trivial in Cl($X'$). Thus
$g^*C_i-n_ig^*D_i$ is a principal divisor, given by a unique (up to scalar) function $h_i'$
in the quotient field of $A:=\mathcal O_{ X',0}$.
Writing $h_i'=j_i/k_i,$ with $j_i,k_i \in A$, we see that
$h_i=j_ik_i^{n_i-1}\ \in A$ is a regular function, defining an
effective divisor of the form $g^*(C_i)+n_iF_i$, for some $F_i$.  Since $g$ is a covering off the singular point, $g^*(C_i)$ is a reduced divisor.  So,
one can form the cyclic cover by adjoining an $n_i$-th root of $h_i$ and
normalizing; this depends only on $h_i'$.   Since the covering group of $X'$ acts on $h_i$ by multiplication by a root of $1$, adjoining a root of $h_i$ keeps the new cover Galois (and abelian) over $(X,0)$.

The other $g^*(C_j)$ share no common components with each other, so remain reduced when pulled back to the cyclic cover given by a root of $h_i$.  Thus we can repeat the cyclic quotient construction, and eventually obtain a series of cyclic covers branched to order $n_j$ over every $g^*(C_j)$, and hence over $C_j$.  The final local ring is the normalization obtained after adjoining to $A$ the $n_i^{th}$ root of $h_i$, all $i=1,\cdots,r$.
\end{proof}
We  summarize in the

\begin{corollary}  Let $(X,C=\sum_{i=1}^rn_iC_i)$ be a singular pair for which $X$ has $\Q$HS link $\Sigma$. Then the universal abelian log cover $(X'',C'')\rightarrow (X,C)$ is constructed by forming the universal abelian cover of $X$, and then taking cyclic covers of orders $n_1$ through $n_r.$  The covering group $G$ sits in a short exact sequence
$$0\rightarrow \Z/(n_1)\oplus \cdots \oplus 
\Z/(n_r)\rightarrow G\rightarrow
H_1(\Sigma;\Z)\rightarrow 0\,.$$ 
\end{corollary}

\begin{example}  In Example (2.6(2)) above, taking a fourth root of the function $z$ gives the UALC of $(X,2C_1)$; thus, the covering group is cyclic, so the exact sequence above is not split.
\end{example}
\begin{example} In Example (2.(6.3)) above, the UALC of $(X,2C_2)$ is obtained by first taking the UAC of $X$; set $x=u^2, y=uv, z=v^2$, and the inverse image of $C_2$ is $v^2-u^4=0$.  Taking the double branched cover gives $\C[u,v,t]/t^2-v^2+u^4$ as UALC, an $A_3$-singularity, with curve the two components of $t=v^2-u^4=0$.  The covering group is the Klein four-group $\Z/2\times \Z/2$, generated by $(t,v,u)\mapsto (-t,v,u)$ and $(t,v,u)\mapsto (t,-v,-u)$.  The invariants of the first involution is the UAC, while that of the second is the $D_4$-singularity described as the ``first case'' in Example (2.(6.3)).  The invariants of the other involution, which is $-I$ on the coordinates, gives the $8/3$ cyclic quotient; in the previous notation, $S=tv$ and $P=tu$.
\end{example}  
\begin{example} Consider the pair consisting of $(\C^2,0)$ and $r\geq
  3$ lines $L_1,\cdots,L_r$ through the origin, with multiplicities $n_1,\cdots,n_r \geq 2$.
  The UALC of  $(\C^2,\sum_{i=1}^r n_iL_i)$ is a Brieskorn complete intersection of type
  $(n_1,\cdots,n_r).$  Specifically, for any $r-2$ distinct numbers
  $a_1,\cdots ,a_{r-2},$ consider the singularity $(Y,0)$ defined
  by $$x_i^{n_i}+x_{r-1}^{n_{r-1}}+a_ix_{r}^{n_r}=0,\ \
  i=1,\cdots,r-2,$$ and the map $f:(Y,0)\rightarrow (\C^2,0)$ defined
  on the ring level by
  $$x\mapsto x_r^{n_r},\ \ y\mapsto x_{r-1}^{n_{r-1}}.$$
 Thus, $y+a_ix$ maps to $-x_i^{n_i}$, so the map is branched over the
 corresponding line with multiplicity $n_i, i=1,\cdots,r-2 $.  It is
 not hard to see that $f$ is the UALC over these lines and the
 coordinate axes, and is the quotient by the group $ \Z/(n_1)\oplus
 \cdots \oplus \Z/(n_r)$ acting diagonally on the coordinates in
 $\C^r$.
\end{example}

   \section{{Orbifold homology group from plumbing}}
 
  As always,  $(\tilde{X},E)\rightarrow (X,0)$ will denote a good resolution of a singularity with $\Q$HS link $\Sigma$, where $E= E_1\cup\cdots \cup E_m$.  From the weighted resolution dual graph $\Gamma$, one forms the lattice $\E=\oplus \Z\cdot E_i$ and its dual $\E^*$, with dual basis $e_1,\cdots,e_m$ defined by $e_i(E_j)=\delta_{ij}$.  The following is well-known:
 
 \begin{proposition} $H_1(\Sigma)$ is isomorphic to the \emph{discriminant group} $D(\Gamma)=\E^*/\E$, with generators  $e_1,\cdots, e_m$, modulo the relations 
 $$\sum_{j=1}^m (E_i\cdot E_j)e_j=0, \ \ i=1,2,\cdots,m.$$
 \end{proposition}
 
 This result follows from examining the
exact sequence
   $$H_2(\Sigma)=0\rightarrow H_2(\tilde{X})\rightarrow  H_2(\tilde{X},\Sigma)\rightarrow H_1(\Sigma)\rightarrow H_1(\tilde{X})=0.$$
Note $H_2(\tilde{X})=\oplus \Z [E_i]\cong \E$ and $H_2(\tilde{X},\Sigma)\cong H^2(\tilde{X})\cong H_2(\tilde{X})^*\cong \E^*$.  More precisely, let $K_i \subset \Sigma$ be a meridian knot over $E_i$, i.e., the boundary of a complex disk $A_i$ in $\tilde{X}$  which is transversal to $E_i$.  Then the classes $[A_i]$ form a basis of $H_2(\tilde{X},\Sigma)$, dual to the $[E_i]$, and the image of $[A_i]$ in $H_1(\Sigma)$ is  $[K_i]$.
 
 The $\Q$-valued pairing on $\E^*$ allows one to define $e_i\cdot e_j \in \Q$; by linear algebra, the matrix $(e_i\cdot e_j)$ is the inverse of the matrix $(E_i\cdot E_j)$.  From (\cite{nw3}, (1.2) and (9.2)), one can calculate the  topologically defined \emph{linking number} $\ell$ of  meridian knots on $\Sigma$:
 
 \begin {proposition}  The linking number satisfies 
 $$\ell(K_i,K_j)=-e_i\cdot e_j.$$ 
 \end{proposition}

A surface pair $(X,C)$ gives rise to an orbifold $(\Sigma, \gamma_i, n_i)$ on the boundary, and we wlll prove one can compute the orbifold homology $H_1^{orb}(\Sigma)$ from a \emph{log resolution} of the pair.  That is, we consider a good resolution  $(\tilde{X},E)\rightarrow (X,0)$ for which the full inverse image of $C$ has strong normal crossings.  A proper transform $\tilde{C}_i$ of $C_i$ intersects one exceptional curve $E_i$ transversally, so that its boundary $\gamma_i$ is a meridian knot of type $K_i$ over  a point of $E_i$; to it is associated an orbifold weight $n_i \geq 1$.  For convenience, one can blow up further so that each exceptional curve intersects at most one $\tilde{C}_i$.  

More generally, assume one is given a good resolution with exceptional curves $E_1,\cdots,E_m$, and a meridian knot $K_i$ and weight $n_i$ associated to each $E_i$.  This defines an orbifold structure $(\Sigma, \gamma_i, n_i)$ on the link. The goal will be to prove the 
 
 \begin{theorem} For the orbifold $(\Sigma, \gamma_i, n_i)$, the orbifold homology $H_1^{orb}(\Sigma)$ is generated by $e_1,\cdots, e_m$, modulo the relations 
 $$\sum_{j=1}^m n_i(E_i\cdot E_j)e_j=0, \ \ i=1,2,\cdots,m.$$
 \end{theorem}

In other words, multiply the rows of the intersection matrix by the corresponding orbifold weight, and take the cokernel of the resulting non-symmetric matrix.  Also, note that if $n_i=1$, then $K_i$ and $\gamma_i$ are irrelevant.

By Lemma 1.5, to prove the Theorem one needs to find first the cohomology of $H_1(\Sigma-\bigcup \gamma_i)$.  By Lemma 1.4 and the above discussion, one has as generators the classes of $K_i$  and $\mu_j$, and the complete set of relations must be of the form
 $$(*)\ \ \ \sum_{j=1}^m(E_i\cdot E_j)[K_j]=\sum_{j=1}^m n_{ij}[\mu_j],\ \ i=1,\cdots,m,$$ for some integral matrix $(n_{ij})$.

Note there is a natural orientation of the $\mu_j$; a small tubular neightborhood $T_j$ of $\gamma_j$ in $\Sigma$ is a complex disk in $E_j$ times $\gamma_j$, and $\mu_j$ is the boundary of a transverse disk.  

Because of Lemma 1.5, Theorem 4.3 follows easily from the next Proposition.  The proof is due to Helge M\o{}ller Pedersen, after which an alternate proof due to Eduard Looijenga will be outlined.

\begin{proposition} With the notation above, the relations in $H_1(\Sigma-\bigcup \gamma_i)$ are
        $$[\mu_i]+\sum_{j=1}^m (E_i\cdot E_j)[K_j]=0, \ \ i=1,\cdots,m.$$
        \begin{proof} We use linking numbers to relate the classes of the $K_i$ and the $\mu_j$.
\begin{lemma} Suppose $K$ is a meridian as above, with $[K]$ of order $d$ in $H_1(\Sigma)$.  Then there is a closed oriented two-chain in $\Sigma-\bigcup\gamma_i$ whose boundary is homotopic to $$dK-\sum_{j=1}^md\ell(K,\gamma_j)\mu_j.$$
\begin{proof} There is a closed oriented 2-chain $A$ in $\Sigma$ whose boundary consists of $d$ copies of $K$.  We may modify $A$ around the intersection points with all the $\gamma_j$ so that there are $d|\ell(K,\gamma_j)|$ transverse intersection points of $A$ with $\gamma_j$ (itself a meridian curve).  As $\ell(K,\gamma_j)$ is positive by Proposition 4.2, $\gamma_j$ intersects $A$ positively at these points.  A small closed tubular neighborhood of $\gamma_j$ intersects $A$ in $d\ell(K,\gamma_j)$ small disks, and the boundary of each is homotopic to $+\mu_j$.  Removing the interior of these disks from $A$ therefore gives a closed $2$-chain in $\Sigma-\bigcup\gamma_i$ with boundary $dK-\sum d\ell(K,\gamma_j)\mu_j$, as claimed.
\end{proof}
\end{lemma}
   Denoting by $\delta=|\det(E_i\cdot E_j)|$ the order of $H_1(\Sigma)$, we multiply by $\delta/d$, and conclude the relations in $H_1(\Sigma-\bigcup \gamma_i)$
   $$\delta [K_i]=\sum_{j=1}^m \delta \ell(K_i,K_j)[\mu_j],\ \ i=1,\cdots,m.$$
   Viewing as a matrix equation, left multiply both sides by $(E_i\cdot E_j)$, which by $(4.2)$ is minus the inverse of  $\ell(K_i,K_j)$; this gives
 $$\delta \sum_{j=1}^m(E_i\cdot E_j)[K_j]=-\delta [\mu_i].$$ 
 Substituting into $(*)$ above yields $$\delta \sum n_{ij}[\mu_j]=-\delta [\mu_i].$$
 But the $[\mu_i]$ are $\Z$-independent in  $H_1(\Sigma-\bigcup \gamma_i)$.  Thus $n_{ij}=-\delta_{ij}$, so that
$$\sum_{j=1}^m(E_i\cdot E_j)[K_j]=-[\mu_i ].$$
  \end{proof}
  \end{proposition} 
  \begin{remark} Looijenga's approach is for each $E_i$ to find the boundary of an appropriate closed oriented $2$-chain in $\Sigma-\bigcup \gamma_i$.  

Suppose $\gamma_i$ lies over $r\in E_i$; $p_1,\cdots, p_t $ are intersection points of $E_i$ with its neighbors; and $q_1,\cdots, q_b$  are any other distinct points of $E_i$.  There is a meromorphic section $s$ of the normal bundle $N\rightarrow E_i$, with simple poles at the $q_k$ and no zeroes.  Let $E_i'$ result from removing from $E_i$ the interiors of small disks centered at $r$, the $p_j$, and the $q_k$; one now has the section $s:E'_i\rightarrow N$, which can be scaled into a section of an appropriate unit circle bundle $N'\subset N$.  Now consider the images in $N'$ of the boundaries of the removed disks.  

As $N'\rightarrow E_i$ is locally $D\times S^1\rightarrow D$, a fibre $K_i$ is homotopic to the inverse image of the boundary of $D$ (consider the family $(te^{i\theta}, e^{i\theta}), 0\leq t\leq 1)$.   Since $s$ has a simple pole at $q_k$, the image of the boundary of $D$ has the opposite orientation as the inverse image (cf. $z\mapsto (z,1/z)\mapsto (z,|z|/z)$).   So the $b$ circles contribute $-bK_i$  to the boundary of $s(E'_i)$.  The circle around $r$ contributes the knot $\mu_i$ to the boundary.  

Since disks were removed from the intersections of $E_i$ with other $E_k$, the 2-chain $s(E'_i)$ can be pushed up from $N'$ minus these disks to $\Sigma-\bigcup\gamma_i$.  Locally $E_i\cap E_k$ looks analytically like $zw=0$, and $\Sigma$ can be viewed as $|zw|=\epsilon$.  The boundary of the disk $\Delta_k=\{|z|\leq \delta, w=0\}$ has boundary that can be lifted to $\{|z|=\delta, w=\epsilon/\delta\}$, which on $\Sigma$ gives a meridian of $E_k$, of class $K_k$.

Therefore, the closed oriented $2$-chain $s(E'_i)\subset \Sigma-\bigcup \gamma_i$ has boundary homotopic to
$$-bK_i+\mu_i+\sum_{k\neq i}(E_i\cdot E_k)K_k=\mu_i+\sum_k (E_i\cdot E_k)K_k.$$
This expression becomes zero in homology.
\end{remark}

 \section{{Action of the orbifold discriminant group}}
 
   With $(\tilde{X},E)$ as before, suppose $E_1, \cdots, E_t$ are the ends (or \emph{leaves}) of the graph $\Gamma$.  Recall the diagonal representation of the discriminant group $D(\Gamma)$:
 
 \begin{proposition}(\cite{nw1}, (5.2), (5.3)).   There is a natural injection $D(\Gamma)\hookrightarrow (\Q/\Z)^t$, given by $e\mapsto (e\cdot e_1, \cdots,e\cdot e_t)$.  Exponentiating $\Q/\Z \hookrightarrow \C^*$ via $r\mapsto \text{exp}(2\pi ir)$, one has a faithful diagonal action of $D(\Gamma)$ on $\C^t$, where the entries are $t$-tuples of $\det (\Gamma)$-th roots of unity.
 \end{proposition}
 
Given a pair $(X,C)$, there is a log resolution so that each $\tilde{C}_i$ intersects an end $E_i$ of the graph (cf. (6.4) below).  For this orbifold situation, to those ends $E_i$ we are given knots $\gamma_i$ of type $K_i$ plus a weight $n_i$; assign a weight $n_j=1$ to any other end.   We represent the extra data by a decorated graph $\Gamma^*$, by adding to each end of $\Gamma$ an arrow and the associated weight.   
\begin{definition}  The orbifold discriminant group $D(\Gamma^*)$ is the group generated by $e_1,\cdots,e_m$ modulo the relations
$$\sum_{j=1}^m n_i(E_i\cdot E_j)e_j=0, \ \ i=1,2,\cdots,t.$$
$$\sum_{j=1}^m (E_i\cdot E_j)e_j=0, \ \ i=t+1,\cdots,m.$$
\end{definition}
According to Theorem 4.3, $D(\Gamma^*)$ is isomorphic to the orbifold homology group (all weights of interior $E_j$ equal $1$).
Modding out the $e_i$'s by the further relations with all $n_i=1$ (as in Proposition 4.1) gives the surjection $\Phi:D(\Gamma^*)\rightarrow D(\Gamma)$, with kernel the sum of cyclic groups of orders $n_1$ through $n_t$.  This is the same as the general result in Proposition 1.1.

$D(\Gamma^*)$ has a natural diagonal representation compatible with that for $D(\Gamma)$, using  the 
 ``power map'' $N:(\C^*)^t\rightarrow (\C^*)^t$ given on points by $$(a_1,\cdots, a_t)\mapsto (a_1^{n_1},\cdots,a_t^{n_t}).$$

\begin{proposition} There is a natural injection $D(\Gamma^*)\hookrightarrow (\C^*)^t$ and a commutative diagram  
       $$D(\Gamma^*)          \hookrightarrow               (\C^*)^t$$          
                       $$  \ \     \downarrow  \Phi     \ \\ \ \ \ \ \     \ \ \                                          \downarrow  N$$
                        $$   D(\Gamma)    \hookrightarrow                                                     (\C^*)^t$$ 
\begin{proof} We claim there is a map $D(\Gamma^*)\rightarrow (\Q/Z)^t$ given by $$e\mapsto ((e\cdot e_1)/n_1, (e\cdot e_2)/n_2,\cdots,(e\cdot e_t)/n_t).$$
One needs to check that the relations defining $D(\Gamma^*)$ go to $0$ in every entry of $(\Q/\Z)^t$. But  in the $k$th entry,
$$\sum_{j=1}^m n_i(E_i\cdot E_j)e_j\mapsto \sum_{j=1}^m n_i(E_i\cdot E_j)e_j\cdot e_k/n_k=n_i/n_k\sum_{j=1}^m (E_i\cdot E_j)e_j\cdot e_k=n_i/n_k \delta_{ik},$$
which equals $0$ or $1$, in either case $0$ in $\Q/\Z$.   After exponentiation, one gets a map to $(\C^*)^t$.  This gives the commutative diagram asserted.

There remains to check injectivity in the top row, or equivalently injectivity of $\ker \Phi\rightarrow \ker N$. As these groups have the same order $n_1\cdots n_t$, we show surjectivity.  But by definition, the map above sends $E_1$ to $(1/n_1,0\cdots,0)$ in $(\Q/Z)^t$, hence to one of the obvious generators of $\ker N$.\end{proof}                        
       \end{proposition}

 \begin{remark}  According to \cite{nw1}, (5.1), the discriminant group $D(\Gamma)$ can be generated by any collection of $t-1$ of the classes $e_1, \cdots, e_t$ of the leaves.  However, $D(\Gamma^*)$ might require more than $t-1$ generators, and even the $t$ leaf classes might not generate.  
  \end{remark}
 \begin{example}  Consider the decorated graph $\Gamma^*$ below of a $D_4$ singularity:
 
 $$
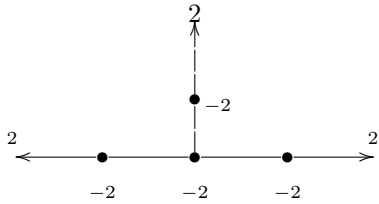
\xymatrix@R=8pt@C=30pt@M=0pt@W=0pt@H=0pt{
&&&&&\\
&&&&2&&&&\\
&&&&\lineto[d]\ar[u]&&\\
&&&&\lineto[d]\\\
&&&&\righttag{\bullet}{-2}{6pt}&&&&\\
&&&&\lineto[u]&&\\
&&\overtag{~}{2}{12pt}&\ar[l]\undertag{\bullet}{-2}{4pt}\lineto[r]&\undertag{\bullet}{-2}{4pt}\lineto[r]\lineto[u]&\undertag{\bullet}{-2}{4pt}\lineto[r]&\overtag{\ar}{2}{12pt}\\
&&\\
&&\\
&&\\
}
$$ 
Index the three leaves clockwise by $e_1,e_2,e_3$, and the central curve by $f$.  Then the discriminant group $D(\Gamma)$ is the direct sum of two cyclic groups of order $2$, generated by $e_1$ and $e_2$; note $e_3=e_1+e_2$, and $f=0)$. However, $D(\Gamma^*)$ is the direct sum of two cyclic groups of order $4$, generated by $e_1$ and $e_2$, and an additional cyclic group of order $2$, generated by $f$ (which is not in the span of $e_1, e_2, e_3$).  So the kernel of the natural projection is $<2e_1>\oplus <2e_2>\oplus <f>$.
 \end{example}  
                  

   On the other hand, we have the following
     
     \begin{proposition}  The orbifold discriminant group $D(\Gamma^*)$ requires at most $t$ generators.
    \begin{proof}  Let $s=m-t$ denote the number of interior vertices of $\Gamma$.    We claim that the second set of $s$ relations in $(5.2)$ imply that $s$ of the $e_i$ can be written as combinations of the $t$ remaining $e_j$, which then suffice as generators of $D(\Gamma^*)$.  
   
    We use the notation $e_i$ to denote the corresponding vertex of $\Gamma$.
    An interior vertex $e_k$ of valence $r$ gives rise to a relation $-(E_k \cdot E_k)e_k+e_{k,1}+\cdots +e_{k,r}=0$, so that any $e_{k,j}$ can be written as  a combination of the $r$ other $e$'s.    In other words, if $e_j$ is any neighbor of $e_k$ , then $e_j$ can be written as a combination of  $e_k$ and $e_k$'s remaining neighbors.  
    
      Choose any interior vertex $f$, which we label the ``center."   We will connect every adjacent pair of vertices with a red or green arrow pointing away from $f$.  First, insert a red arrow from $f$ to one of its neighbors $f'$.  From $f$ to any of its other neighbors insert a green arrow.  Next, if the end of an arrow is an end of $\Gamma$, do nothing more at that vertex.  Otherwise, at any vertex which is the end of an arrow, choose a red arrow to one of its (other) neighbors; if there  are any further neighbors, choose a green arrow to each of them.  Continue until one reaches all the end vertices of $\Gamma$.   In this way, for every interior vertex, there is one red arrow emerging.  Call the $s$ vertices which are the end of red arrows ``red vertices", while the remainder (there are $t$ of them) are ``green vertices."     
       
      We claim that each red vertex can each be written as a combination of the green ones.  Every vertex has a distance from the center, so we prove the result by induction on the distance.  We have already said that the center $v$ is green, as are all but one of its neighbors.  In particular, its red neighbor can be written as a combination of $v$ and $v$'s green neighbors.   Next, let $e$ be any red vertex.  Then it is the end of a red arrow emanating from a vertex $e'$ which is closer to the center.  Then $e$ is a combination of $e'$ and $e'$'s other neighbors.  The other neighbors of $e'$ are $e''$, which is  closer to the center, and possibly extra ones (if the valency of $e'$ is at least $3$).  But those extra neighbors are green, since only one red arrow can emerge from $e'$, and that is the one to $e$.  Thus, $e$ can be written as a combination of green vertices plus $e'$ and $e''$.  But $e'$ and $e''$ are closer to the center, so the inductive step applies (whether or not either is red or green).      
     \end{proof}
     \end{proposition}

\bigskip      
       
\section{{orbifold splice quotients}}
  
  We briefly outline the basics of splice quotient singularities; see Section $1$ of \cite{nw3} for a few more details, and \cite{nw1} for the full story.  
   
       We retain the same notation as before for a good resolution of a singularity, including the graph $\Gamma$ and the diagonal representation $D(\Gamma)\hookrightarrow( \C^*)^t$.  $\Gamma$ has leaves (vertices of valence 1) and nodes (vertices of valence $\geq 3$).  Associated to $\Gamma$ is a \emph{splice diagram} $\Delta$, obtained by collapsing all vertices of valence two, and assigning a positive integer weight to every node and emanating edge by taking an appropriate determinant.  From this data, for each node one can assign a weight to every leaf of $\Delta$; the \emph{semigroup condition} on $\Delta$ then requires that each weight on an emanating edge of that node is in the semigroup generated by the weights of the ``outer'' leaves.   Choosing a coordinate $x_i$ for each of the $t$ leaves, one can then write a total of $t-2$ \emph{splice diagram equations} in the $x_i$'s, by writing for each node a sum of monomials (one for each emanating edge) which have the same weight as the product of weights around the node (one uses the coefficients of the semigroup relations).  The \emph{congruence condition} allows one to insure that each equation so obtained transforms by a character under the action of the representation of $D(\Gamma)$ on $\C^t$.   One thus produces from $\Gamma$ isolated complete intersection singularities in $\C^t$ on which $D(\Gamma)$ acts, freely off the origin.  There are choices involved, and higher order terms can be added. 
       
         $\Gamma$ is called \emph{quasi-minimal} if any string in the graph contains no $-1$ vertex, or consists of a unique $-1$ vertex.  (A \emph{string} is a connected subgraph containing no nodes.) The Main Theorem $7.2$ of \cite{nw1} starts with a graph $\Gamma$:
       
       \begin{theorem}\cite{nw1}  Suppose $\Gamma$ is quasi-minimal and satisfies the semigroup and congruence conditions.  Then:
       \begin{enumerate}
       \item A set of splice diagram equations $\{f_j(x_i)=0\}$ for $\Gamma$ defines an isolated complete intersection singularity $(X',0)\subset (\C^t,0)$.
       \item The discriminant group $D(\Gamma)$ acts freely on $X'-\{0\}$.
       \item The quotient $(X,0)\equiv (X',0)/D(\Gamma)$ has an isolated normal surface singularity, and a good resolution $(\tilde{X},E)\rightarrow (X,0)$ whose associated dual graph is $\Gamma$.
       \item $f:(X',0)\rightarrow (X,0)$ is the universal abelian cover, in particular unramified off the origins.
       \item Each curve $C_i'=\{x_i=0\}\cap X'$ is mapped by $(X',0)\rightarrow (X,0)$ to an irreducible curve $C_i$, whose proper transform $\tilde{C_i}$ on $\tilde{X}$ is smooth and intersects the exceptional curve transversally, along an end $E_i$.  
       \end{enumerate}
       \end{theorem}
       Such $(X,0)$ are called \emph{splice quotient singularities}.
       
        In more detail, the curve  $C_i'$ is reduced, with $|D(\Gamma)|/h$ irreducible components, where $h$ is the order of $e_i$ in $D(\Gamma)$ (follows from \cite{nw1}, Section 3).  $x_i^h$ is $D(\Gamma)$ invariant, hence a function on $(X,0)$; its zero-set is a Cartier divisor, which is $h$ times an irreducible Weil divisor $C_i$.     $x_i^h$ is an \emph{end-curve function} for $(X,0)$; on the resolution $(\tilde{X},E)$ it vanishes only on $E$ and an \emph {end-curve $\tilde{C_i}$}, which intersects $E_i$ transversally at one point.  Thus the above construction gives not only a splice-quotient singularity $(X,0)$, but also a collection of end-curve functions and their corresponding irreducible curves $C_i$.  Note finally that $f^*(C_i)$ is the reduced sum of the irreducible components of $C'_i$.
                        
        Whether a given $(X,0)$ is a splice quotient singularity is given by the principal result of \cite{nw3},  the \emph{End-Curve Theorem}:
        
        \begin{theorem} \cite{nw3} Let $(X,0)$ be a normal surface singularity with $\Q$HS link, $(\tilde{X},E)\rightarrow (X,0)$ a good resolution.  Suppose for every end $E_i$ of the exceptional curve $E$ there is a function on $(X,0)$ whose zero set on $\tilde{X}$ is an end-curve for $E_i$.  Then $(X,0)$ is a splice quotient singularity.
        \end{theorem}
        
        An immediate Corollary is the result of T. Okuma:
        
        \begin{theorem}\cite{okuma2}  Let $(\tilde{X},E)\rightarrow (X,0)$ be a quasi-minimal resolution of a rational surface singularity.  Then the graph $\Gamma$ satisfies the semigroup and congruence conditions.  Moreover, every end-curve $\tilde{C_i}$ is cut out by an end-curve function on $X$.
        \begin{proof}   For any end-curve $\tilde{C_i}$ on $\tilde{X}$, the image $C_i$ on $X$  is a $\Q$-Cartier divisor (since $X$ is rational), so some multiple of it is the zero-set of a function on $X$.
        \end{proof}
        \end{theorem}
 
 To analyze a pair $(X,C)$, one needs a well-adapted log resolution.
   \begin{proposition}  A pair $(X,C=\sum n_iC_i)$ has a  
       \emph{minimal orbifold resolution}, a smallest log resolution $\pi:(\tilde{X_C},E)\rightarrow (X,0)$ satisfying
       \begin{enumerate}
       \item  each $C_i$ has  proper transform $\tilde{C_i}$ which is smooth, intersecting $E$ transversally
       \item $\tilde{C_i}$ intersects a leaf $E_i$ of $E$
       \item each leaf $E_j$ of $E$ intersects at most one $\tilde{C_i}$
       \item $\tilde{X_C}$ is quasi-minimal.
       \end{enumerate}       
       $\tilde{X_C}$ is unique except when a log resolution consists of a single curve plus two $\tilde{C}_i$ intersecting it transversally at two points.
        \begin {proof} Starting with the minimal good resolution of $X$, resolve in a minimal way the singularities of the reduced inverse image of each $C_i$ until one has strong normal crossings (the minimal log-resolution).  If the proper transform of some $C_i$ does not intersect a leaf, blow-up that intersection point.   If three or more proper transforms intersect the same leaf, blow-up each of the intersection points.  If two intersect the same leaf, which has another connection with the graph, again blow-up those points.  The only missing case is when two transforms intersect a graph consisting of a single curve; in that case, one could blow-up either of the intersection points.  Otherwise, one has uniqueness of the resolution.

          We claim the corresponding resolution is quasi-minimal.  If a $-1$ curve intersects some $\tilde{C_i}$, then it is an end-curve; if its neighboring curve had valency $2$, the $-1$ curve could be blown down, contradicting the minimality of the blow-up process.  If a $-1$ curve intersects no $\tilde{C_i}$ and it has a neighbor curve with valency $2$, again one could contract the $-1$ curve, violating the minimality process.
        \end{proof}
        \end{proposition}
  We exclude the simple cases for which one has non-uniqueness of the resolution.   
     
Thus, a pair gives rise to a quasi-minimal graph $\Gamma_C$, as well as a decorated version $\Gamma_C^*$ which has arrows and weights at special edges.   From this data, we can assign weight one to non-special edges, and encode all the weights in a power map  $N:\C^t\rightarrow \C^t$ given on functions by $x_i\mapsto z_i^{n_i}$.   By $(5.3)$, one has  $D(\Gamma^*)=N^{-1}(D(\Gamma)).$

We give conditions on graphs $\Gamma$ and $\Gamma^*$ that allow one to give equations for a pair $(X,C)$ and its UALC.  This is Theorem 2 of the introduction, which we now prove.

Let $\Gamma$ be a quasi-minimal graph with $t$ ends satisfying the semigroup and congruence conditions, and $\Gamma^*$ a decorated version with arrows and weights at special ends.  Assign a coordinate $x_i$ to each end of $\Gamma$.  Choose a complete intersection singularity $(X',0)\subset (\C^t,0)$ defined by $t-2$ splice equations $\{f_k(x_1,x_2,\cdots,x_t)=0\}$ on which $D(\Gamma)$ acts, freely off $0$.  Then $(X',0)\rightarrow (X',0)/D(\Gamma)\equiv (X,0)$ is the UAC, and there is a resolution $(\tilde{X},E)\rightarrow(X,0)$ with graph $\Gamma$.  If $E_i$ is a special end of $E$,  the curve $C'_i=\{x_i=0\}\cap X'$ is reduced and maps to a $\Q$-Gorenstein curve $C_i$ on $(X,0)$. Let $C'=\Sigma C_i'$.  Consider as above a power map $N:\C^t\rightarrow \C^t$. 

\begin{theorem} Consider $\Gamma$, $\Gamma^*$, $(X',0)$,  $f_k(x_1,\cdots,x_t)$, $(X,0)$, $C_i$, $C'$, $N:\C^t\rightarrow \C^t$,  as above.
\begin{enumerate}
\item Define $(X'',0)= N^{-1}(X',0)=\{f_k(z_1^{n_1},\cdots,z_t^{n_t})=0,\  1\leq k \leq t-2\}\subset \C^t$. Then $(X'',0)$ is an isolated complete intersection singularity on which $D(\Gamma^*)$ acts.
\item Let $C''\subset X''$ be the reduced Cartier divisor which is the sum of $\{z_i=0\}\cap X''$ for special $i$.  Then $D(\Gamma^*)$ acts on $C''$ and acts freely on $X''-C''$.
\item The quotient $(X'',C'')\rightarrow (X'',C'')/D(\Gamma^*)= (X,C)$ is the universal abelian log cover, where $C=\Sigma n_iC_i$ is the sum (with the weights) corresponding to the special ends of $\Gamma^*$.
\end{enumerate}
\begin{proof}  We show first that $(X'',0)$ is non-singular away from the origin.  $N$ is a covering map off the intersection with the coordinate hyperplanes $z_i=0$.  $N$ can be factored as the composition of maps which raise powers one coordinate at a time.  To study what happens over the reduced curve $\{x_i=0\}\cap X'$, factor the map $N$ by first raising the $i^{th}$ coordinate to the $n_i^{th}$ power, and then raising all the other coordinates to the appropriate power.  The inverse image of  $\{x_i=0\}\cap X'$ under the first map is smooth (away from $0$), and the second map is unramified over the new curve.  Further,  $\{z_i=0\}\cap X''$ is reduced.

Note $D(\Gamma^*)$ is a subgroup of $(\C^*)^t$, which itself preserves all  $\{z_j=0\}$, so it acts on $X''$ as well as $C''$.  By construction, $N:(X'',0)\rightarrow (X',0)$ is a covering map off $C''$.  Since $D(\Gamma)$ acts freely on $X'-\{0\}$, then $N^{-1}(D(\Gamma))=D(\Gamma^*)$ acts freely on $X''-C''$.  The full (and free) quotient of $X''-C''$ by $D(\Gamma^*)$ is thus $X-\Sigma C_i$.  The construction above matches that of the UALC as described in Section 3.
\end{proof}
\end{theorem}

\begin{definition} A pair which arises as in the Theorem is called an \emph{orbifold splice quotient}.
\end{definition}
\begin{remark}  Note that the condition for $(X,\Sigma n_iC_i)$ to be a splice quotient is independent of the weights $n_i$, and depends only on the $C_i$ (which must be $\Q$-Cartier).
\end{remark}

There is an analogue of the End-Curve Theorem for orbifold splice quotients.

\begin{theorem} Let $(X,C=\Sigma n_iC_i)$ be a pair, with minimal orbifold resolution and data $(\tilde{X},E)$, $\Gamma$, $\Gamma^*$.  Then the following are equivalent:
\begin{enumerate}
\item $(X,C)$ is an orbifold splice quotient
\item Each $C_i$ is $\Q$-Cartier, and for every non-special end $E_j$ of $E$ there is an end-curve function.
\end{enumerate}
\begin{proof}  One implication follows from the above discussion; we prove the converse.  

Since $C_i$ is $\Q$-Cartier, it is the zero-locus of a function $g_i$ on $X$.  Therefore, on $\tilde{X}$, $g_i$ vanishes only on $E$ and $\tilde{C_i}$, hence is an end-curve function for that special leaf.  But by assumption, for each non-special leaf there is an end-curve function.  Thus, there are end-curve functions for every end of $\Gamma$.  By the End-Curve Theorem, $\Gamma$ satisfies the semigroup and congruence conditions and $(X,0)$ itself is a splice quotient.  It follows from the construction above that $(X,C)$ is an orbifold splice quotient.
\end{proof}
\end{theorem}

\begin{corollary}  If $(X,0)$ has a rational singularity, then any pair $(X,C)$ is an orbifold splice quotient.
\begin{proof}  Since $(X,0)$ is rational, every curve $C_i$ is $\Q$-Cartier.  For each non-special end $E_j$ on the minimal orbifold resolution $\tilde{X}$, choose any end-curve $\tilde{C_j}$.  Its image $C_j$ on $X$ is the zero-set of a function $h_j$, which becomes an end-curve function of $E_j$.  The result now follows from the Theorem.
\end{proof}
\end{corollary}

Therefore, if $(X,0)$ has a rational singularity and $C\subset X$ is a reduced curve, then modulo taking an abelian quotient,  one can write down from the orbifold resolution diagram $\Gamma$ some explicit equations for the singularity and the components of the curve $C$.  This includes the case $X=\C^2$, where it has been known for a long time (e.g.  \cite{e-n}) how to write down the equation of a singular curve from its resolution diagram.  One should compare as well with Example $3.5$.

\begin{corollary} Consider the pair $(X,\Sigma n_iC_i)$.   Suppose $(X,0)$ is a splice quotient, each $C_i$ is $\Q$-Cartier, and all the non-special ends of $E$ on the minimal orbifold resolution are ends from the minimal good resolution.  Then $(X,C)$ is an orbifold splice quotient.
\begin{proof}  Since $(X,0)$ is a splice quotient, there is an end-curve function for every end of the graph of the minimal good resolution.  Such a function works just as well if that end is also an end on the minimal orbifold resolution.   Now apply the Theorem.
\end{proof}
\end{corollary}


       If $(X,0)$ is a splice-quotient singularity which is not rational, there will likely exist quasi-minimal resolutions for which the semigroup conditions are not satisfied.
        
                  \begin{example}  The minimally elliptic hypersurface singularity $(X,0)$=$\{x^2+y^3+z^7=0\}$ is a splice quotient, with graph and splice diagram for the minimal good resolution       
       $$
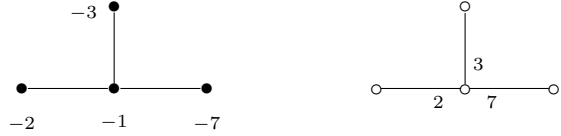
\xymatrix@R=8pt@C=30pt@M=0pt@W=0pt@H=0pt{
\\&&\lefttag{\bullet}{-3}{4pt}&&&&\Circ\\
\\
\\
&\undertag{\bullet}{-2}{4pt}\lineto[r]&\undertag{\bullet}{-1}{3pt}\lineto[r]\lineto[uuu]&\undertag{\bullet}{-7}{4pt}&&\Circ\lineto[r]_(.7){2}&\Circ\lineto[r]_(.3){7}\lineto[uuu]_(.3)3&\Circ\\\\
\\
}
$$ 

Blowing up further gives a quasi-minimal resolution graph and splice diagram:

    $$
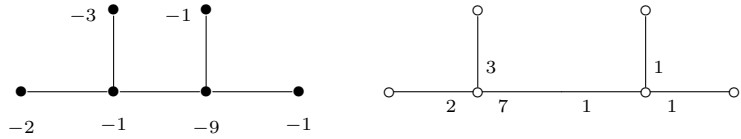
\xymatrix@R=8pt@C=30pt@M=0pt@W=0pt@H=0pt{
\\&&\lefttag{\bullet}{-3}{4pt}&\lefttag{\bullet}{-1}{3pt}&&&\Circ&&\Circ\\
\\
\\
&\undertag{\bullet}{-2}{4pt}\lineto[r]&\undertag{\bullet}{-1}{3pt}\lineto[r]\lineto[uuu]&\undertag{\bullet}{-9}{4pt}\lineto[uuu]\lineto[r]&\undertag{\bullet}{-1}{3pt}&\Circ\lineto[r]_(.7){2}&\Circ\lineto[r]_(.3){7}\lineto[uuu]_(.3)3&\lineto[r]_(.3){1}&\Circ\lineto[r]_(.3){1}\lineto[uuu]_(.3){1}&\Circ\\\\
\\
}
$$ 
The splice diagram does not satisfy the semigroup conditions.  
      \end{example} 
      One can understand the previous example as follows: A singularity with this graph (there are two of them) is minimally elliptic and has unimodular resolution graph.  Thus the divisor class group is isomorphic to $\C$, so there is no torsion, and $\Q$-Cartier is the same as Cartier.  The only end-curve functions for the MGR on the $-7$ curve are powers of $z$; so on the further blow-up, one cannot have end-curve functions for both the new $-1$ end-curves.

\bigskip

\section {{An example}}
Let $(X,0)$ be the $D_4$ singularity  $\{w^2=xy(x+y)\}$.  Consider the curve $C=pC_1+qC_2$, where $C_1$ is the Weil divisor given by the prime ideal $(y,w)$, and $C_2$ is a curve, equations to be determined, which on the minimal resolution is smooth and intersects tangentially the central curve, to order $2$, away from the other  intersection points.  To go from the familiar minimal resolution of $X$ to an orbifold resolution, one needs to blow up three times so that $\tilde{C_2}$ intersects a unique leaf.  The graph $\Gamma_C^*$ is
 $$
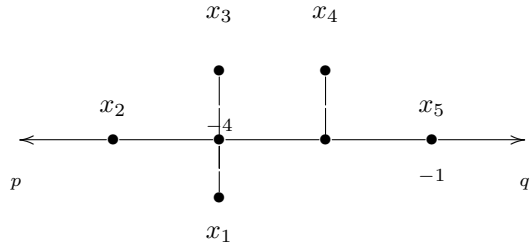
\xymatrix@R=8pt@C=30pt@M=0pt@W=0pt@H=0pt{
&&&&&\\
&&&&&&&&&\\
&&&&x_3&x_4&\\
&&&&&&\\
&&&&\righttag{\bullet}{}{6pt}&\righttag{\bullet}{}{6pt}&&&\\
&&&x_2&\lineto[u]&\lineto[u]&x_5&\\
&&\undertag{~}{p}{12pt}&\ar[l]\undertag{\bullet}{}{4pt}\lineto[r]&\overtag{\bullet}{-4}{6pt}\lineto[r]\lineto[u]\lineto[d]&\undertag{\bullet}{}{4pt}\lineto[r]\lineto[u]&\undertag{\bullet}{-1}{4pt}\lineto[r]&\undertag{\ar}{q}{12pt}\\
&&&&\lineto[u]&&\\
&&&&\undertag{\bullet}{}{6pt}\lineto[u]&&\\
&&&&x_1&\\
}
$$      
The representation on $\C^5$ of $D(\Gamma)$ is the four-group, acting only on the $x_1,x_2,x_3$ coordinates, by multiplying two of the variables by $-1$.  There will be $2$ splice diagram equations coming from the valence $4$ node (locating the cross-ratio of the fourth point), and one from the other.  One set of equations (there are many) on which the group acts equivariantly is
$$x_4-x_2^2-x_3^2=0$$
$$x_1^2+x_2^2-x_3^2=0$$
$$x_5-x_4^2-x_1x_2x_3(x_1^2+x_2^2+x_3^2)=0.$$   
Invariants for the group action are $x=x_1^2, y=x_2^2, z=x_3^2, w=x_1x_2x_3$, and $x_4$ and $x_5$, and these are related by
$$x_4=y+z;\ x+y-z=0;\ x_5=x_4^2+w(x+y+z).$$
Using $xyz=w^2$, the quotient is the familiar equation $w^2=xy(x+y)$.  The curve $C_2$ turns out to be a principal (i.e., Cartier) divisor defined by
$x_5=(x+2y)^2+2w(x+y).$  $x_2=0$ defines $C_1'$, so $x_2^2=y$ is an end-curve function.

To pass to the UALC, in all equations one replaces $x_2$ by $z_2^p$,  $x_5$ by $z_5^q$, and $x_1,x_3,x_4$ by $z_1,z_3,z_4$ respectively.  The orbifold homology group in $(\C^*)^5$ is generated by $(1/2)[1,0,1,0,0]; (1/2p)[p,1,0,0,0]$; and $ (1/q)[0,0,0,0,1],$  hence is the product of three cyclic groups of orders $2$, $2p$, and $q$, respectively.  Here we have used the familiar notation $(1/n)[q_1,\cdots,q_k]$ to denote the diagonal group generated by $(\zeta^{q_1},\cdots, \zeta^{q_k})$, where $\zeta$ is a primitive $n^{th}$ root of $1$.

\end{document}